\documentclass{amsart}
\usepackage[utf8]{inputenc}
\usepackage[francais]{babel}
\usepackage{tikz,amssymb}
\usetikzlibrary{calc}

\newcommand{\Z}{\ensuremath{\mathbf{Z}}}
\newcommand{\R}{\ensuremath{\mathbf{R}}}

\theoremstyle{plain}
\newtheorem{thm}{Theorem}
\newtheorem{lemma}[thm]{Lemma}

\title{On norms taking integer values on the integer lattice}

\author{Mikael de la Salle}
\thanks{Supported by ANR grant GAMME}
\address{CNRS-ENS de Lyon, UMPA UMR 5669
\newline F-69364 Lyon cedex 7, France}
\email{mikael.de.la.salle@ens-lyon.fr}

\keywords{Thurston norm, Convex body ; Norme de Thurston, Corps convexe}
\subjclass{}

\begin{document}

\begin{abstract}
We present a new proof of Thurston's theorem that the unit ball of a seminorm on $\R^d$ taking integer values on $\Z^d$ is a polyhedra defined by finitely many inequalities with integer coefficients.
\vskip 0.5\baselineskip

On pr\'esente une nouvelle preuve du th\'eor\`eme de Thurston selon lequel la boule unit\'e d'une seminorme sur $\R^d$ prenant des valeurs enti\`eres sur $\Z^d$ est un poly\`edre d\'efini par un nombre fini d'in\'egalit\'es \`a coefficients entiers.
\end{abstract}

\maketitle

\section{Version Française}

Dans l'étude de ce qu'on appelle maintenant la \emph{norme de Thurston} \cite{MR823443}, Thurston démontre que la boule unité d'une semi-norme $N$ sur $\R^d$ prenant des valeurs entières sur $\Z^d$ est un polyèdre défini par un nombre fini d'inégalités à coefficients entiers. Le but de cette note est de donner une preuve différente et à mon avis plus directe de ce théorème. On travaille avec la boule unité duale $B^*$, et on démontre que ses points exposés sont contenus dans $\Z^d$, ce qui suffit puisque $B^*$ est l'enveloppe convexe fermée de ses points exposés par le théorème classique de Straszewicz \cite{straszewicz}. Soit donc $y_0 \in B^*$ un point exposé. Cela signifie qu'il est l'unique point d'intersection de $B^*$ avec un hyperplan d'appui de $B^*$, dont on appelera $x_0$ un vecteur normal, normalisé pour que $\langle x_0,y_0\rangle =1$. Un petit raisonnement géométrique (Figure \ref{picture:dessin1} et formule (\ref{eq:toutestla})) permet de se convaincre que, pour un point $x$ de $\R^d$, $N(x)$ diffère de $\langle x,y_0\rangle$ d'une quantité tendant vers $0$ lorsque $x$ s'éloigne de l'origine en restant à distance bornée de la demi-droite dirigée par $x_0$. En appliquant ceci à une suite de tels points $x_n \in \Z^d$ et à $x_n+e_j$ avec $e_j$ le $j$-ème vecteur de la base canonique, on en déduit que $\langle e_j,y_0\rangle = \langle x_n+e_j,y_0\rangle - \langle x_n,y_0\rangle$ est arbitrairement proche de l'entier $N(x+e_j)-N(x)$, donc est entier, ce qu'il fallait démontrer.

\subsection*{Remerciements} Je remercie Pierre Dehornoy pour d'intéressantes discussions et pour ses encouragements à écrire cette note.

\begin{figure}
  \center
\begin{tikzpicture}[scale=1.3]
\begin{scope}[rotate=36,xscale=2]
  \draw (0,0) circle (1cm);
  \draw (0,0)--(-20:1cm) node[above left] {$y_0$} node {$\bullet$};
  \draw[loosely dotted] (0,0)--(-20:2.5cm);
  \draw[shift=(-20:1cm),rotate=-20] (-90:2cm)--(90:.8cm);
  \draw[dashed] (0,0)--(-35:1cm) node {$\bullet$} node[left] {$y_n$};
  \draw[dashed, shift=(-35:1cm),rotate=-35] (-90:1cm)--(90:1cm);
  \end{scope}
  \coordinate  (U) at ($ (-65:2cm) + (5cm,0)$);
  \draw[dotted,-latex] (0,0)-- node[near end, above] {$\lambda_n {x}_0$} (5cm,0);
  \draw[dotted,-latex] (5cm,0)-- node[above right] {${z}_n$} (U) node[below] {$x_n$};
  \draw[loosely dotted] ($(115:2.3cm) + (5cm,0)$)--(5cm,0);
  \draw (0,0)--(U);
\end{tikzpicture}
\caption{All lines which look orthogonal are orthogonal.} \label{picture:dessin1}
\end{figure}
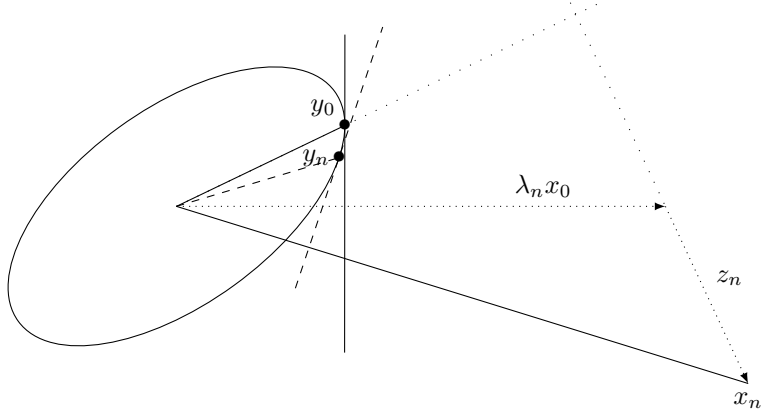

\section{English version}

Consider $\R^d$ with its scalar product $\langle x,y\rangle = \sum_{i=1}^n x_i y_i$ and associated norm $\|x\|$. The purpose of this note is to give an alternate proof of the following theorem of Thurston \cite[Theorem 2]{MR823443}.
\begin{thm}[Thurston]\label{thm:thurston} If $N$ is a seminorm on $\R^d$ taking integer values on $\Z^d$, then there is a finite subset $F \subset \Z^d$ such that $N(x) = \max_{y \in F} \langle x,y\rangle$ for all $x \in \R^d$.
\end{thm}
This theorem is important in connexion with the Thurston norm \cite{MR823443}, which is a seminorm on the second homology group $H_2(M;\R)$ of an oriented $3$-manifold $M$ taking, by construction, integer values on $H_2(M;\Z)$. Theorem \ref{thm:thurston} allows to deduce the important feature of this seminorm that its unit ball is a polyhedron defined by linear inequalities with integer coefficients. 

Thurston's proof of Theorem \ref{thm:thurston} is short and quite easy, but not very enlightening (at least to me). The proof presented here is direct, by showing that the set of exposed points of the dual unit ball $B^* = \{y \in \R^d, \langle x,y\rangle \leq N(x) \forall x \in \R^d\}$
is contained in $\Z^d$. The set $B^*$ is a convex compact subset of $\R^d$ symmetric around $0$, with non-empty interior if and only if $N$ is a norm. By Hahn--Banach $B^*$ allows to recover $N$ by \begin{equation}\label{eq:dual_norm} N(x) = \max_{y \in B^*} \langle x,y\rangle.\end{equation}
A point $y_0 \in B^*$ is called exposed if there is a supporting hyperplane which intersects $B^*$ at $y_0$ only, or equivalently if there is $x_0 \in \R^d$ which \emph{exposes $y_0$}, \emph{i.e.} such that $\langle x_0,y\rangle < \langle x_0,y_0\rangle$ for every $y \in B^* \setminus \{y_0\}$. We shall use the following characterization of exposed points, which asserts that the $N$-seminorm of a point far in the direction of $x_0$ is almost attained at $y_0$.
\begin{lemma}\label{lem:exposed} Let $C$ be a compact convex set and $y_0 \in C$ be a point exposed by $x_0 \in \R^d$. If $x_n \in \R^d$ is such that $\sup_n \|x_n - n x_0\| <\infty$, then $\lim_n N(x_n) - \langle x_n,y_0\rangle = 0$.
\end{lemma}
Before we prove this lemma, let us explain how it implies Thurston's theorem. Let $y_0 \in B^*$ be a point exposed by $x_0$. For $1 \leq j \leq d$, let $e_j \in \Z^d$ be the $j$-th coordinate vector. Pick $x_n \in \Z^d$ the closest point (in euclidean distance) to $nx_0$. By the lemma and the fact that $N(x_n) \in \Z$ we have $\lim_n d(\langle x_n,y_0\rangle,\Z) =0$. Similarly by the lemma applied to $x_n + e_j$, we have $\lim_n d(\langle x_n+e_j,y_0\rangle,\Z) =0$, and hence
\[ d(\langle e_j,y_0\rangle,\Z) \leq \lim_n d(\langle x_n,y_0\rangle,\Z) + d(\langle x_n+e_j,y_0\rangle,\Z) = 0.\]
This proves that $\langle e_j,y_0\rangle \in \Z$.

Since the preceding is valid for every $j$ and every exposed point $y_0$, we have proved that the set $\mathrm{exp}(B^*)$ of \emph{exposed points} of $B^*$ is contained in $\Z^d$; it is finite because it is bounded. We conclude the proof by the classical theorem by Straszewicz \cite{straszewicz} that, as every convex compact subset of $\R^d$, $B^*$ is the closed convex hull of $\mathrm{exp}(B^*)$, and by \eqref{eq:dual_norm} $N(x) = \sup_{y \in \mathrm{exp}(B^*)} \langle x,y\rangle$.

\begin{proof}[Proof of Lemma \ref{lem:exposed}]
We can normalize $x_0$ so that $\langle x_0,y_0\rangle=1$ (unless $\langle x_0,y_0\rangle=0$, which implies that $C=\{0\}$). In particular $N(x_0)=1$.

Decompose $x_n = \lambda_n x_0 + z_n$ with $z_n \perp y_0$, so $\lambda_n  = \langle x_n,y_0\rangle \leq N(x_n)$. Also, $z_n$ is bounded because $x_n$ stays at bounded distance from $n x_0$ and $\lambda_n \sim n$ is positive for $n$ large. 

Let $y_n \in C$ such that $N(x_n) = \langle x_n,y_n\rangle$. See Figure \ref{picture:dessin1}. Since $x_0$ exposes $y_0$ and $C$ is compact we have  $\lim_n \|y_n - y_0\|=0$ (if $y' \in C$ is an accumulation point of the sequence $y_n$, then we have $N(x_0)  = \lim_n N(x_n/n) = \lim_n \langle x_n/n,y_n\rangle = \langle x_0,y'\rangle$, so that $y' = y_0$ because $x_0$ exposes $y_0$). Using that $\langle x_0,y_n\rangle \leq 1$ and that $\langle z_n,y_n\rangle = \langle z_n,y_n-y_0\rangle$, we therefore get by the Cauchy--Schwarz inequality
\begin{equation}\label{eq:toutestla} \lambda_n \leq N(x_n) =  \lambda_n \langle x_0,y_n\rangle + \langle z_n,y_n\rangle \leq  \lambda_n + \|z_n\|\|y_n-y_0\| = \lambda_n+o(1).\end{equation}
\end{proof}


\begin{thebibliography}{99}
\bibitem{straszewicz}
S.~{Straszewicz}.
\newblock {\"Uber exponierte Punkte abgeschlossener Punktmengen.}
\newblock {\em {Fundam. Math.}}, 24:139--143, 1935.


\bibitem{MR823443}
W.~P. Thurston.
\newblock A norm for the homology of {$3$}-manifolds.
\newblock {\em Mem. Amer. Math. Soc.}, 59(339):i--vi and 99--130, 1986.

\end{thebibliography}
\end{document}